\documentclass[reqno]{amsart}

\begin{document}

\title{}

\title[\hfilneg \hfil Almost automorphic solution]
{Almost automorphy and Riccati Equation}

\author[Indira Mishra \hfil \hfilneg]
{ Indira Mishra \\
The Institute of Mathematical Sciences, Chennai\\
CIT Campus, Taramani, Chennai, 600 113. India.}

\address{Indira Mishra \newline
The Institute of Mathematical Sciences, Chennai
CIT Campus, Taramani, Chennai, 600 113. India.}

\email{indira.mishra1@gmail.com, indiram@imsc.res.in}

\date{}

\begin{abstract}
In this paper we first consider a linear time invariant systems with almost periodic
forcing term. We propose a new deterministic quadratic control problem, motivated by
Da-Prato. With the help of associated degenerate Riccati equation we study the existence
and uniqueness of an almost automorphic solutions.
\end{abstract}

\maketitle \numberwithin{equation}{section}
\newtheorem{theorem}{Theorem}[section]
\newtheorem{lemma}[theorem]{Lemma}
\newtheorem{example}[theorem]{Example}
\newtheorem{proposition}[theorem]{Proposition}
\newtheorem{corollary}[theorem]{Corollary}
\newtheorem{remark}[theorem]{Remark}
\newtheorem{definition}[theorem]{Definition}

\vskip .3cm \noindent
\section{ \textbf{Introduction}}
The study of periodic systems, periodic optimization problems and the existence of
almost periodic solutions to quadratic control problems have been quite interesting
topic to work for many Mathematicians for e.g. \cite{pr1, a1, pr2, s1, v1} and
\cite{pr3}.
In \cite{pr2}, Da-Prato has considered the quadratic control problem for the periodic
systems in infinite dimension and showed that optimal control is given by feedback
control which involves periodic solution of Riccati equation.

In \cite{pr1}, Da-Prato and Ichikawa considered the following system
\begin{eqnarray}
y'=Ay+Bu+f \label{0}
\end{eqnarray}
where $u\in L^2_{ap}(U), \ U$ and $Y$ are real separable Hilbert spaces, $B\in
\mathcal{L}(U,Y)$ and $f\in L^2_{ap}(Y)$.
Under the assumptions that the pair $(A,B)$ is stabilizable and $(M,A)$ is detectable,
they studied the existence of almost periodic solutions for (\ref{0})
and minimized the cost functional:
\begin{eqnarray}
J(u)=\lim_{T\rightarrow \infty}\frac{1}{T}\int_0^T (|My|^2+|u|^2)dt \nonumber
\end{eqnarray}
over the set of all controls $u$ such that (\ref{0}) has an almost periodic
mild solution.

In this paper we study the same problem as considered in \cite{pr1}, by weakening
their hypothesis. The conditions $(A,B)$ stabilizable $(M,A)$ detectable is somehow
more demanding, whereas similar results as in \cite{pr1} can be obtained with
the assumption on only one pair that the pair $(-A,B)$ is exactly controllable.
We generalize their results \cite{pr1} for the case of almost automorphic
solutions that is we wish to minimize the cost functional considered over the set of
all controls such that equation (\ref{0}), with almost automorphic forcing term has
an almost automorphic solution.

Almost automorphy is the generalization of almost periodicity. The notion of almost
automorphic functions were introduced by S. Bochner \cite{boch} in relation to some
aspects of differential geometry.

Let $Y$ and $U$ be Hilbert spaces.
We consider the following differential equation:
\begin{eqnarray}
\dot{y}=Ay+Bu+f \label{eq1}\\
y(0)=y_0 \nonumber
\end{eqnarray}
where $A$ is the infinitesimal generator of semigroup $(T(t))_{t\ge 0}$ on $Y$. Let
$B\in \mathcal{L}(U,Y), \ u\in L^2_{aa}(Y)$ and $f\in L^2_{aa}(Y)$ where
$L^2_{aa}(Y)$ is the completion of $AA(Y)$ with respect
to the inner product
$$\lim_{T\rightarrow \infty}\frac{1}{T}
\int_0^T \langle f(t),g(t)\rangle dt$$
defined on $AA(Y)$, which is discussed in detail in Section \ref{s3}.

We can not always expect $J(u)$ to be finite. We wish to minimize the average cost
functional defined by
\begin{eqnarray}
J(u)=\lim_{T\rightarrow \infty}\frac{1}{T}\int_0^T \Big(|My|^2+|u|^2\Big)dt \label{1}
\end{eqnarray}
over $\mathcal{U}_{ad}$.
\begin{eqnarray}
\mathcal{U}_{ad}=\{u\in L^2_{ap}(Y): \ \lim_{T\rightarrow \infty}\frac{1}{T}\int_0^T
\Big(|My|^2+|u^2|\Big)dt<\infty \} \nonumber
\end{eqnarray}
where $M\in \mathcal{L}(Y)$. We write (\ref{1}) as
\begin{eqnarray}
J(u)=|My|^2_{aa}+|u|^2_{aa}. \nonumber
\end{eqnarray}

The optimization problem (\ref{eq1}) with cost functional
(\ref{1}) is more general than the following optimal control problem.
\begin{eqnarray}
\dot{y}=Ay+Bu \label{eq2}\\
y(0)=y_0 \nonumber
\end{eqnarray}
Define
\begin{eqnarray}
J(u)=\frac{1}{2}\int_0^\infty |My|^2+\frac{1}{2}\int_0^\infty |u|^2 \label{2}
\end{eqnarray}
The linear regulator problem is to minimize $J(u)$ where $(y,u)$ satisfy (\ref{eq2}).
\\

\noindent {\bf Finite Cost Condition:} For every $y_0\in Y, \ u\in L^2(0,\infty;U)$
such that $J(u)$ defined in (\ref{2}) is finite. and $y$ satisfy (\ref{eq2}).
Then we know by \cite{zab} that there exists a bounded linear operator $P=P^* \ge 0$
which satisfies the following operator algebraic Riccati equation:
\begin{eqnarray}
A^* P+PA-PBB^* P+M^* M=0 \nonumber
\end{eqnarray}
In this paper we are interested in making use of degenerate algebraic Riccati equation.
That is there exists $P=P^*\ge 0$ in $\mathcal{L}(Y)$ such that
\begin{eqnarray}
A^* P+PA-PBB^* P=0 \nonumber
\end{eqnarray}

\section{Preliminaries} \label{s2}
In this Section we recall some basic definitions, Lemmas and results
going to be used in the sequel. Let $(Y,\|\cdot\|)$ be a Hilbert space.
\subsection{Almost Periodic Functions}
\begin{definition}
A subset $S \ \mbox{of} \ \mathbb{R}$ is called relatively dense if
there exists a positive number $l>0,$ such that $[a,a+l]\cap S\neq
\phi,$ for all $a\in \mathbb{R}.$
\end{definition}
\begin{definition} see \cite{fink},
We say that a function $f$ is Bohr almost periodic, if for every
$\epsilon>0,$ the set $T_{f,\epsilon}$ given by
$$T_{f,\epsilon}=\{\tau : \sup_{t\in \mathbb{R}}
|f(t+\tau)-f(t)|<\epsilon\},$$ is relatively dense for all $t\in
\mathbb{R}$.
\end{definition}
We denote by $AP(X),$ the set of all almost periodic functions
defined on $X.$ The space $(AP(X),\|\cdot\|_{AP(X)})$ is a Banach
space with the supremum norm given by,
\begin{eqnarray}
\|u\|_{AP(X)}=\sup_{t\in \mathbb{R}}\|u\|_X. \nonumber
\end{eqnarray}

\subsection{Almost automorphic functions}
\begin{definition}
A continuous function $f:\mathbb{R}\rightarrow Y$ is called almost automorphic
if for every sequence $(s'_n)_{n\in \mathbb{N}}$ of real numbers,
there exists a subsequence $(s_n)_{n\in \mathbb{N}}$ such that
\begin{eqnarray}
\lim_{n,m\rightarrow \infty}\|f(t+s_n-s_m)-f(t)\|=0. \nonumber
\end{eqnarray}
\end{definition}
We denote by $AA(Y)$ the set of all almost automorphic functions, defined on $Y$.
\begin{remark}
$AA(Y)$ with respect to the $\sup$-norm is a Banach space.
\end{remark}
\begin{example}
Concrete examples of almost automorphic functions are:
\begin{eqnarray}
a(t)=\cos\bigg(\frac{1}{\cos t+\cos(\sqrt 2t)}\bigg)~\mbox{and}~
~b(t)=\sin\bigg(\frac{1}{\sin t+\sin(\sqrt 5t)}\bigg) \nonumber
\end{eqnarray}
$t\in \mathbb{R}$.
\end{example}
\subsection{Hyperbolic Semigroups}
Let $(Y, \|\cdot\|)$ be a Hilbert space.
If $A$ be a linear operator on $Y$, then $\rho(A), \ \sigma(A), \ D(A),
 \ R(A), \ \ker(A)$ denote the resolvent, spectrum, domain, range and kernel of
the operator $A$. Let $Y_1, \ Y_2$ be two Hilbert spaces, then
$\mathcal{L}(Y_1,Y_2)$ is the Banach space of bounded linear operators from
$H_1$ into $H_2$, when $Y_1=Y_2$ this is simply denoted by $\mathcal{L}(Y)$.
\begin{definition} See \cite{toka}.
A linear operator $A: D(A)\subset Y\rightarrow Y$ (not necessarily densely
defined) is said to be sectorial if the following hold:
there exist constants $\zeta\in \mathbb{R}, \ \theta\in (\frac{\pi}{2},\pi)$
and $M>0$ such that $S_{\theta,\zeta}\subset \rho(A)$,
\begin{eqnarray}
S_{\theta,\zeta}:=\{\lambda\in \mathbb{C}: \lambda\neq \zeta, \
\arg(\lambda-\zeta)<\theta\}, \nonumber\\
\|R(\lambda,A)\|\le \frac{M}{|\lambda-\zeta|}, \quad
\lambda\in S_{\theta,\zeta} \nonumber
\end{eqnarray}
where $R(\lambda,A)=(\lambda I-A)^{-1}$ for each $\lambda\in \rho(A)$.
\end{definition}
\begin{remark} See \cite{toka}.
If the operator $A$ is sectorial, then it generates an analytic semigroup
$(T(t))_{t\ge 0}$, which maps $(0,\infty)$ into $\mathcal{L}(Y)$ and such
that there exists constants $M_0, M_1>0$ such that
\begin{eqnarray}
&&\|T(t)\|\le M_0e^{\zeta t}, \quad t>0\\
&&\|t(A-\zeta I)T(t)\|\le M_1e^{\zeta t}, \quad t>0
\end{eqnarray}
\end{remark}
\begin{definition} See \cite{toka}.
A semigroup $(T(t))_{t\ge 0}$ is said to be hyperbolic if there exist
projection $\Pi_s$ and constants $M, \ \delta>0$ such that $T(t)$
commutes with $\Pi_s, \ \ker \Pi_s$ is invariant with respect to $T(t), \
T(t): R(\Pi_u)\rightarrow R(\Pi_u)$ is invertible and
\begin{eqnarray}
\|T(t)\Pi_sx\|\le M e^{-\delta t}\|x\|, \ t\ge 0\\
\|T(t)\Pi_ux\| \le Me^{\delta t}\|x\|, \ t\le 0
\end{eqnarray}
where $\Pi_u:=I-\Pi_s$ and for $t\le 0, \ T(t):=(T(-t))^{-1}$.
\end{definition}

\begin{remark}
The existence of a hyperbolic semigroup on a Banach space $X$ gives us a
nice algebraic information about this vectorial space. In fact let
$(T(t))_{t\ge 0}$ be a hyperbolic semigroup on $X$. Then there are $(T(t))_{t\ge 0}$
invariant closed subspaces $X_s$ and $X_u$ such that $X=X_s\oplus X_u$.
Furthermore the restricted semigroups $(T_s(t))_{t\ge 0}$ on $X_s$ and
$(T_u(t))_{t\ge 0}$ on $X_u$ have the following properties:
\begin{itemize}
\item[(i)] The semigroup $(T_s(t))_{t\ge 0}$ is uniformly exponentially
stable on $X_s$.
\item[(ii)] The operators $T_u(t)$ are invertible on $X_u$ and $(T_u(t)^{-1})_{t\ge 0}$
is uniformly exponentially stable on $X_u$.
\end{itemize}

\end{remark}
We recall that analytic semigroup $(T(t))_{t\ge 0}$ associated with the linear
operator $A$ is hyperbolic if and only if $\sigma(A)\cap{i\mathbb{R}}=\phi$.

\begin{definition} See \cite{toka}.
For $\alpha\in (0,1)$ Banach space $(Y_\alpha, \|\cdot\|_\alpha)$ is said to
an intermediate space between $D(A)$ and $Y$ if
$D(A)\subset Y_\alpha\subset Y$
and there is a constant $c>0$ such that
\begin{eqnarray}
[x]_\alpha=\sup_{0\le t\le 1}\|t^{1-\alpha}(A-\zeta I)e^{-\zeta t}T(t)x\|<\infty, \nonumber
\end{eqnarray}
with the norm
\begin{eqnarray}
\|x\|_\alpha=\|x\|+[x]_\alpha \label{eq}
\end{eqnarray}
and the abstract Holder spaces $D_A(\alpha):=\overline{D(A)}^{\|\cdot\|_\alpha}$.
\end{definition}
\begin{lemma} See \cite{toka}. \label{l2.8}
For a hyperbolic analytic semigroup $(T(t))_{t\ge 0}$ there exists constants $C(\alpha)>0,
\delta>0, \ M(\alpha)>0$ and $\gamma>0$ such that
\begin{eqnarray}
&&\|T(t)\Pi_ux\|_\alpha\le c(\alpha)e^{\delta t}\|x\| \mbox{ \ for \ } t\le 0\\
&&\|T(t)\Pi_sx\|_\alpha \le M(\alpha)t^{-\alpha}e^{-\gamma t}\|x\| \mbox{ \ for \ } t>0.
\end{eqnarray}
\end{lemma}

\section{Quadratic Control Problem} \label{s3}
In this section we consider linear systems described by a strongly continuous semigroup
and solve quadratic control problems.

Let $Y$ be a real separable Hilbert space with inner product $\langle,\rangle$
and norm $|\cdot|$. Let $f(t)$ be a continuous function in $Y$ and $f$ is said to be
almost automorphic if for every sequence $(s_n')_{n\in \mathbb{N}}$ of real numbers
there exists a subsequence $(s_n)_{n\in \mathbb{N}}$ such that $f(t)=\lim_{n\rightarrow
\infty} f(t+s_n-s_m)$. Almost automorphic functions are the generalization of almost
periodic functions. An almost periodic function is almost automorphic but the converse
is not true. For example:
\begin{eqnarray}
f(t):=\sin\Big(\frac{1}{\sin t+\sin \sqrt{2}t}\Big) \nonumber
\end{eqnarray}
is almost automorphic but not almost periodic. We make a note that almost
automorphic functions are bounded. It is easy to see that set of all almost
automorphic functions (scalar functions) $AA(\mathbb{R})$ forms an algebra.
The set of all almost automorphic functions defined on $Y$ will be denoted by
$AA(Y)$. Let $f,g\in AA(Y)$. Then the mean value
$\lim_{T\rightarrow \infty}\int_0^T|f(t)|^2dt$ exists.
Thus the limit
$$\lim_{T\rightarrow \infty}\frac{1}{T}
\int_0^T \langle f(t),g(t)\rangle dt$$
defines an inner product on $AA(Y)$
which we denote by $\langle f,g \rangle_{aa}$. The corresponding norm
$|\cdot|_{aa}$. Let $L^2_{aa}(Y)$ be the completion of $AA(Y)$ with respect
to this inner product.

Now we consider the differential equation:
\begin{eqnarray}
y'=Ay+f \label{q}
\end{eqnarray}
where $A$ is the infinitesimal generator of strongly continuous semigroup $e^{tA}$
on $Y$ and $f\in AA(Y)$. If $Y=\mathbb{R}^n$ and $A\in \mathbb{R}^{n\times n}$
then $e^{tA}$ is the usual matrix exponential function. If $A$ is stable that is
$e^{tA}$  is exponentially stable, then
\begin{eqnarray}
y(t)=\int_{-\infty}^t e^{(t-s)A}f(s)ds \label{q1}
\end{eqnarray}
is well defined and is almost automorphic. In fact let $(a_n)$ be an arbitrary
sequence then
\begin{eqnarray}
y(t+a_n)&=&\int_{-\infty}^{t+a_n} e^{(t+a_n-s)A}f(s)ds \nonumber\\
&=&\int_{-\infty}^t e^{(t-\tau)A} f(t+a_n)d\tau \nonumber
\end{eqnarray}
In general $y(t)$ given by (\ref{q1}) does not satisfy (\ref{q}) but we can
find a subsequence $f_n\in AA(Y)$ such that $f_n\rightarrow f$ uniformly in $AA(Y)$
and $y_n(t)=\int_{-\infty}^t e^{(t-s)A}f_n(s)ds$
is the solution of (\ref{q}) with $f=f_n$, uniformly converging to $y$ in $AA(Y)$.
A continuous function $y$ on $\mathbb{R}$ is called mild solution of (\ref{q}) if
\begin{eqnarray}
y(t)=e^{(t-s)A}y(s)+\int_s^t e^{(t-r)A}f(r)dr \nonumber
\end{eqnarray}
for any $t\ge s$. Then $y(t)$ given by (\ref{q1}) is unique mild solution of (\ref{q})
in $AA(Y)$. In fact if $z$ is another solution, then
\begin{eqnarray}
z(t)-y(t)=e^{(t-s)A}[z(s)-y(s)] \nonumber
\end{eqnarray}
Letting $s\mapsto -\infty$ and using the fact that $z$ and $y$ are bounded, we obtain
that $z(t)-y(t)=0$ for any $t$. Letting $s\rightarrow -\infty$ and noting that
$z,y$ are bounded, we obtain $z(t)-y(t)=0$ for any $t$.

Moreover if $e^{tA}$ satisfies exponential dichotomy that is there exists a projection
operator $\pi$ such that
\begin{itemize}
\item[(i)] $Y_1:= \Pi Y\subset D(A)$ and $A_1:=A\Pi$ is a bounded operator on $Y_1$
with
\begin{eqnarray}
|e^{-tA_1}|\le M_1 e^{-a_1 t}, \ t\ge 0 \mbox{ \ for \ some \ } M_1>0, \ a_1>0. \nonumber
\end{eqnarray}
\item[(ii)] $A:D(A)\cap Y_2\rightarrow Y_2,$ where $Y_2:=(I-\Pi)Y$ and $A_2:=A(I-\Pi)$
generates an exponentially stable semigroup on $Y_2$. Then
\begin{eqnarray}
y(t)=\int_{-\infty}^t e^{(t-s)A_2}(I-\Pi)f(s)ds-\int_t^\infty e^{(t-s)A_1}\Pi f(s)ds\nonumber
\end{eqnarray}
is a unique mild solution of (\ref{q}) in $AA(Y)$.
\end{itemize}
\section{Almost automorphic solutions}
We consider the following system:
\begin{eqnarray}
y'=Ay+Bu+f  \label{m1}
\end{eqnarray}
It is not clear that there exists any admissible control $u$
for problem (\ref{eq1}) which minimizes (\ref{2}) over $\mathcal{U}_{ad}$.
Due to presence of $f$, it is reasonable to assume the existence of some
$F\in \mathcal{L}(Y,U)$ such that $A-BF$ is stable and the feedback law is
given by
\begin{eqnarray}
u=-Fy+v, \quad v\in L^2_{aa}(U) \nonumber
\end{eqnarray}
is admissible and in fact
\begin{eqnarray}
y(t)=\int_{-\infty}^t e^{(t-s)(A-BF)}\Big(Bv(s)+f(s)\Big)ds \nonumber
\end{eqnarray}
is the unique mild solution of (\ref{eq1}) in $AA(Y)$.

If we assume that $-A$ is exponentially stable and that the pair $(-A,B)$ is exactly
controllable, then it follows by \cite[Theorem 2.4, pp. 212]{zab} that there exists
a constant $\beta>0$ such that for arbitrary $y\in Y$, we have
\begin{eqnarray}
\int_0^\infty\|B^*e^{-tA^*}y\|^2_Udt\ge \beta\|y\|^2_Y. \nonumber
\end{eqnarray}
We have the following result due to Kesavan \cite{kes1}.
\begin{theorem} \label{t3.1}
Let $U$ and $Z$ be the Hilbert spaces. Let $A:D(A)\subset Z\rightarrow Z$ be
the infinitesimal generator of a $C_0$-semigroup. Let $B\in \mathcal{L}(U,Z)$.
Then the following are equivalent.\\
(i) $-A$ is the exponentially stable and there exists $\alpha>0$ such that
for all $z\in Z$,
\begin{eqnarray}
\int_0^\infty\|B^*e^{-tA^*}z\|^2_Udt\ge \alpha\|z\|^2_Z. \nonumber
\end{eqnarray}
(ii) There exists $P\in \mathcal{L}(Z)$ which is a solution of the degenerate
Riccati equation (\ref{eq3})
\begin{eqnarray}
A^* P+PA-PBB^* P=0 \label{eq3}
\end{eqnarray}
which has the following properties.
\begin{itemize}
\item[(a)] $P=P^*\ge 0$;
\item[(b)] $P$ is invertible;
\item[(c)] $A-BB^*P$ is exponentially stable.
\end{itemize}
\end{theorem}
\begin{remark}
In view of above Theorem $A-BB^*P$ is exponentially stable, thus
it follows from \cite[Proposition 2.1]{pr1} that $r(t)$
given by
\begin{eqnarray}
r(t)=\int_t^\infty e^{(s-t)(A^*-PBB^*)}Pf(s)ds \label{3}
\end{eqnarray}
is almost periodic function lying in $AP(Y)$ and is the unique mild solution of
\begin{eqnarray}
r'(t)+(A^*-PBB^*)Pf(s)ds .\label{eq4}
\end{eqnarray}
\end{remark}
\begin{remark}
As $P$ is a bounded linear operator; let $\|P\|=p$.
\end{remark}
Before stating our main results of this paper, let us make the following
assumptions.\\
{\bf Assumptions:}
\begin{itemize}
\item[H1.]  Assume that $-A$ is exponentially stable.
\item[H2.] The pair $(-A,B)$ is exactly controllable.
\end{itemize}

\begin{theorem}
Assume that (H1) and (H2) holds.
Then there exists a unique optimal control for (\ref{eq1}) and (\ref{1}) and it is
given by the feedback law
\begin{eqnarray}
\bar{u}=-B^*(Py+r)
\end{eqnarray}
and
\begin{eqnarray}
J(\bar u)=2\langle r,f\rangle_{aa}-|B^*r|^2_{aa} \label{4}
\end{eqnarray}
where $P$ and $r$ are unique solutions of (\ref{eq3}) and (\ref{eq4}) respectively.
\end{theorem}
\begin{proof}
A proof for this Theorem is similar to that of \cite[Theorem 2.1]{pr1}
But for reader's convenience we give it here once again.
It follows by Theorem \ref{t3.1} that $A-BB^*P$ is exponentially stable,
so $r(\cdot)$ given by equation (\ref{3}) is the unique mild solution of
(\ref{eq4}),
hence $\bar u(\cdot)$ is admissible. Let $u$ be an arbitrary admissible
control and $y$ its response. We differentiate $\langle Py(t),y(t)\rangle
+2\langle r(t),y(t)\rangle$ formally and remove the terms involving $A$
using (\ref{eq3}). Then we integrate from $0$ to $T$, we obtain
\begin{eqnarray}
&&\langle Py(T),y(T)\rangle+2\langle r(T),y(T)\rangle
-\langle Py(0),y(0)\rangle-2\langle r(0),y(0)\rangle\nonumber\\
&=& -\int_0^T|u|^2dt+\int_0^T |(u+B^*(Qy+r))|^2dt
+\int_0^T 2\langle r(t),y(t)\rangle
-|B^*r|^2 dt \nonumber
\end{eqnarray}
Dividing by $T$ and letting $T$ tending to infinity. We obtain
\begin{eqnarray}
J(u)=|(u+B^*(Qy+r))|^2_{aa}+2\langle r,f\rangle_{aa}-|B^*r|^2_{aa} \label{5}
\end{eqnarray}
where we have used the boundedness of $y(T)$ and $r(T)$. This formal
position can be justified by introducing approximating systems of
(\ref{q}) and (\ref{eq4}) with strict solutions \cite{barbu} and then passing to
the limit. See [3], [9] for arguments based on Yosida approximations of
$A$. Now the optimality of $\bar u$ and (\ref{4}) follows easily from (\ref{5}).
Since $A-BB^*P$ is stable, the uniqueness of $\bar u$ follows.
\end{proof}

Let us denote by $L:=A-BB^*P$, if the semigroup $(T_L(t))_{t\ge 0}$ generated by
$L$ is hyperbolic; that is $\sigma(L)\cap i\mathbb{R}=\phi$
\begin{lemma} \label{l4.4}
Given that the semigroup $(T_L(t))_{t\ge 0}$ generated by $L$
is hyperbolic that is $\sigma(L)\cap i\mathbb{R}=\phi,$ then there
exist a unique almost automorphic solution given by:
\begin{eqnarray}
r(t)=\int_{-\infty}^te^{(t-s)L^*}\Pi_s Pf(s)ds-\int_t^\infty e^{(t-s)L^*}\Pi_uPf(s)ds
\end{eqnarray}
of (\ref{eq4}) in $AA(Y)$, whenever $f\in L^2_{aa}(Y)$.
\end{lemma}
\begin{proof}
We consider
\begin{eqnarray}
r(t+s_n-s_m)-r(t)=
\int_{-\infty}^te^{(t-s)L^*}\Pi_s P\Big(f(s+s_n-s_m)-f(s)\Big)ds\nonumber\\
-\int_t^\infty e^{(t-s)L^*}\Pi_u P\Big(f(s+s_n-s_m)-f(s)\Big)ds \nonumber
\end{eqnarray}
Letting $s\mapsto (t-s)$ in the right hand side; we get
\begin{eqnarray}
r(t+s_n-s_m)-r(t)&=&
\int_0^{\infty}e^{sL^*}\Pi_s P\Big(f(t-s+s_n-s_m)-f(t-s)\Big)ds\nonumber\\
&&-\int_{-\infty}^0 e^{sL^*}\Pi_u P\Big(f(t-s+s_n-s_m)-f(t-s)\Big)ds \nonumber
\end{eqnarray}
Evaluating using Cauchy-Schwarz inequality
\begin{eqnarray}
&&\|r(t+s_n-s_m)-r(t)\|_\infty \nonumber\\
&\le& \|r(t+s_n-s_m)-r(t)\|_{aa} \nonumber\\
&\le& \Big(\int_0^\infty Npe^{-\delta s}ds\Big)
\Big(\lim_{T\mapsto \infty}\frac{1}{T}\int_0^T\|f(t-s+s_n-s_m)-f(t-s)\|^2_Yds\Big)^{\frac{1}{2}}\nonumber\\
&&+\Big(\int_{-\infty}^0 Npe^{\delta s}ds\Big)
\Big(\lim_{T\mapsto \infty}\frac{1}{T}\int_0^T\|f(t+s+s_n-s_m)-f(t+s)\|_Y^2ds\Big)^{\frac{1}{2}} \nonumber\\
&\le& \frac{Np}{\delta}(\|f(t-s+s_n-s_m)-f(t-s)\|_{aa}+\|f(t+s+s_n-s_m)-f(t+s)\|_{aa}) \nonumber\\
&\le& \frac{2Np}{\delta}\varepsilon:=\varepsilon' \nonumber
\end{eqnarray}
This implies that $r\in AA(Y)$ is an almost automorphic function.
\end{proof}
Rewriting equation (\ref{eq1}) as:
\begin{eqnarray}
\dot{y}&=&Ay-BB^*Py-BB^*r+f \nonumber\\
&=&(A-BB^*P)y-BB^*r+f   \nonumber\\
&=& Ly+\Big(f(t)-BB^*r(t)\Big) \nonumber
\end{eqnarray}
Equation (\ref{eq1}) can be written as:
\begin{eqnarray}
\dot{y}(t)=Ly+\Big(f(t)-BB^*r(t)\Big) \nonumber
\end{eqnarray}
\begin{definition}
A bounded continuous function $y: \mathbb{R}\rightarrow Y$ is called mild solution to
equation (\ref{eq1}) if
\begin{eqnarray}
y(t)=T_L(t-s)y(s)+\int_s^t T_L(t-s)\Big(f(s)-BB^*r(s)\Big) ds \nonumber
\end{eqnarray}
$t\ge s, \ s\in \mathbb{R}$.
\end{definition}
\begin{theorem}
Suppose that $-A$ is exponentially stable and the pair $(-A,B)$ is
exactly controllable and the semigroup $(T_L(t))_{t\ge 0}$ generated by $A-BB^*P$
is hyperbolic semigroup. Then for $f\in L^2_{aa}(Y)$, there exists a unique mild
solution to (\ref{eq1}) given by
\begin{eqnarray}
y(t)=\int_{-\infty}^t T_L(t-s)\Pi_s\Big(f(s)-BB^*r(s)\Big)ds-
\int_t^{\infty} T_L(t-s)\Pi_u\Big(f(s)-BB^*r(s)\Big)ds\nonumber
\end{eqnarray}
for all $t\in \mathbb{R}$, where $\Pi_s$ and $\Pi_u$ are the projections associated
to the operator $L$. Furthermore, this mild solution belongs to $AA(Y_\alpha)$.
\end{theorem}
\begin{proof}
\begin{eqnarray}
y(t+s_n-s_m)-y(t)&=&\int_{-\infty}^{t+s_n-s_m} T_L(t+s_n-s_m-s)\Pi_s(f(s)-BB^*r(s))ds \nonumber\\
&&-\int_{-\infty}^t T_L(t-s)\Pi_s(f(s)-BB^*r(s))ds\nonumber\\
&&-\int_{t+s_n-s_m}^\infty T_L(t+s_n-s_m-s)\Pi_u(f(s)-BB^*r(s))ds \nonumber\\
&&+\int_t^{\infty} T_L(t-s)\Pi_u(f(s)-BB^*r(s))ds\nonumber
\end{eqnarray}
Using transformations $s\mapsto -s+t+s_n-s_m$ in $1^{st}$ and $3^{rd}$ terms
and $s\mapsto -s+t$ in $2^{nd}$ and $4^{th}$ terms in right hand side;
\begin{eqnarray}
y(t+s_n-s_m)-y(t)&=&\int_0^{\infty}T_L(s)\Pi_s\bigg(\Big(f(-s+t+s_n-s_m)-f(-s+t)\Big)\nonumber\\
&&-BB^*\Big(r(-s+t+s_n-s_m)-r(-s+t)\Big)\bigg)ds \nonumber\\
&&-\int_{-\infty}^0 T_L(s)\Pi_u\bigg(\Big(f(-s+t+s_n-s_m)-f(-s+t)\Big)\nonumber\\
&&-BB^*\Big(r(-s+t+s_n-s_m)-r(-s+t)\Big)\bigg)ds \nonumber
\end{eqnarray}
We have that $\|B\|=\|B^*\|:=K$ (lets say), for some constant $K>0$, follows
from \cite[Proposition 7.2.1]{kes2}.
Hereafter we use the following conventions
\begin{eqnarray}
\|x\|_{\infty,\alpha}:=\sup_{t\in \mathbb{R}}\|x\|_\alpha \mbox{ \ and \ }
\|x\|_{\alpha,(aa)}=\lim_{T\rightarrow \infty}\frac{1}{T}\int_0^T\|x\|^2_\alpha\nonumber.
\end{eqnarray}
Now we consider the $\alpha$-norm defined in (\ref{eq}) of $\Big(y(t+s_n-s_m)-y(t)\Big)$
and using the inequalities in Lemma \ref{l2.8}, we have:
\begin{eqnarray}
&&\|y(t+s_n-s_m)-y(t)\|_{\infty,\alpha} \nonumber\\
&\le&\|y(t+s_n-s_m)-y(t)\|_{\alpha,(aa)} \nonumber\\
&\le& M(\alpha)\Bigg[\Big(\int_0^\infty s^{-\alpha}e^{-\gamma s}ds\Big)
\Big(\lim_{T\rightarrow \infty}\frac{1}{T}\int_0^T\|f(t-s+s_n-s_m)-f(t-s)\|_Y^2ds\Big)^{\frac{1}{2}}\nonumber\\
&&+ \|B\|^2\|r(t-s+s_n-s_m)-r(t-s)\|_{AA(Y)}\Bigg]\nonumber\\
&&+c(\alpha)\Bigg[ \Big(\int_{-\infty}^0e^{\delta s}ds\Big)
\bigg(\lim_{T\rightarrow \infty}\frac{1}{T}\int_0^T\|f(t-s+s_n-s_m)-f(t-s)\|
_Y^2ds\bigg)^{\frac{1}{2}}\nonumber\\
&&+ \|B\|^2\|r(t-s+s_n-s_m)-r(t-s)\|_{AA(Y)}\Bigg]\nonumber
\end{eqnarray}

With the help of Lemma \ref{l4.4} we have
\begin{eqnarray}
&&\|y(t+s_n-s_m)-y(t)\|_\alpha \nonumber\\
&\le& M(\alpha)(\gamma^{\alpha-1}\Gamma(1-\alpha)\|f(t-s+s_n-s_m)-f(t-s)\|
_{aa}+K^2\varepsilon')\nonumber\\
&&+\frac{c(\alpha)}{\delta}(\|f(t-s+s_n-s_m)-f(t-s)\|_{aa}+K^2 \varepsilon'). \nonumber\\
&\le& M(\alpha)(\gamma^{\alpha-1}\Gamma(1-\alpha)\varepsilon+K^2\varepsilon')
+\frac{c(\alpha)}{\delta}(\varepsilon+K^2 \varepsilon'). \nonumber
\end{eqnarray}
Which shows that $y$ lies in the space $AA(Y_\alpha)$; and hence we have the
desired result.
\end{proof}
Here we consider a simple example which is similar as considered in \cite{pr1}.
\begin{example}
A deterministic example. \\

\noindent Let us take $Y=U=\mathbb{R}$ and $A=3, \
 B=4, \ M=1$ and $f(t)=\sin t$. Then equation (\ref{m1}) is:
\begin{eqnarray}
y'=3y+4u+\sin t. \nonumber
\end{eqnarray}
The solution of (\ref{eq3}) which is nonnegative is $P=\frac{1}{2}$.
Then $r(t)=\frac{1}{52}[\cos t+5 \sin t]$.
It is easy to obtain
\begin{eqnarray}
2\langle r, f\rangle_{aa}=\frac{5}{52}, \quad |B^*r|^2_{aa}=\frac{4}{52}. \nonumber
\end{eqnarray}
Thus the optimal control is given by
\begin{eqnarray}
\bar{u}=-2y-\frac{1}{13}(\cos t+5\sin t) \nonumber
\end{eqnarray}
and
\begin{eqnarray}
J(\bar u)=\frac{1}{52}. \nonumber
\end{eqnarray}
Here $f$ is periodic, but we may add $\sin\sqrt 2t$
then it becomes an almost periodic function and we may also
calculate $\bar u$ and $J(\bar u)$ similarly by choosing $f$ an
almost automorphic function given by:
\begin{eqnarray}
f(t)=\sin{\frac{1}{2+\cos t+\cos \sqrt 2 t}} . \nonumber
\end{eqnarray}
\end{example}

\end{document}